\documentclass{amsart}

\usepackage{latexsym}
\usepackage{amsmath}
\usepackage{amsfonts}
\usepackage{amssymb}
\usepackage{graphicx}
\usepackage{color}
\usepackage[latin1]{inputenc}
\newtheorem{theorem}{Theorem}[section]
\newtheorem{lemma}[theorem]{Lemma}
\newtheorem{proposition}[theorem]{Proposition}
\newtheorem{corollary}[theorem]{Corollary}

\theoremstyle{definition}
\newtheorem{definition}[theorem]{Definition}
\newtheorem{example}[theorem]{Example}

\theoremstyle{remark}

\numberwithin{equation}{section}

\begin{document}

\title{Dual maps and the  Dunford-Pettis property}

\author{Francisco J Garc\'{i}a-Pacheco}
\address{Department of Mathematics, University of Cadiz, Puerto Real, 11510, Spain}
\curraddr{} \email{garcia.pacheco@uca.es}

\author{Alejandro Miralles}
\address{Departament de Matemàtiques, Universitat Jaume I and IMAC, Castelló de la Plana, E-12071, Spain}
\curraddr{} \email{mirallea@uji.es}
\thanks{The second author is supported by Project MTM 2011-22457 (MECC. Spain) and Project P1-1B2014-35 (UJI. Spain).}
\author{Daniele Puglisi}
\address{Department of Mathematics, University of Catania, Catania 95125, Italy}
\email{dpuglisi@dmi.unict.it}
\thanks{}

\subjclass[2000]{Primary 46B20, 46B10}

\date{}

\dedicatory{}

\keywords{}

\pagestyle{plain}

\begin{abstract}
We characterize the points of $\left\|\cdot\right\|$-$w^*$ continuity of dual maps, turning out to be the smooth points.
We  prove that a Banach space has the Schur property if and only if it has the Dunford-Pettis property and there exists a dual map that is sequentially $w$-$w$ continuous at $0$. As consequence, we show the existence of smooth Banach spaces on which the dual map is not $w$-$w$ continuous at $0$. 
\end{abstract}

\maketitle

\section{Introduction}

In \cite[Theorem 1, p. 22]{D1}, it was proved the following result: 
\begin{theorem}
Let $X$ a Banach space and $x_0 \in S_X$. The following assertions are equivalent:
\begin{enumerate}
\item $x_0 \in \mathrm{smo}\left(\mathsf{B}_X\right)$.
\item Every support mapping is weak to norm continuous at $x_0$.
\item There exists a support mapping which is weak to norm continuous at $x_0$.
\end{enumerate}
\end{theorem}

In Section \ref{sec3}, we will prove a more general version of this result by using dual maps instead of support mappings and we will obtain a new characterization of smoothness. Following the study of dual maps, we will prove that a Banach space has the Schur property if and only if it has the Dunford-Pettis property and there exists a dual map that is sequentially $w$-$w$ continuous at $0$ in Section \ref{sec4}.

\section{Background}

A Banach space $X$ is said to have the {\em Dunford-Pettis property (DPP)} provided that for every sequence $x_{n}\overset{w }{\longrightarrow }0$ in $X$ and $x_{n}^{\ast }\overset{w }{\longrightarrow }0$ in $X^{\ast }$, the sequence $\left( x_{n}^{\ast }\left( x_{n}\right) \right) _{n\in \mathbb{N}}$ converges to $0$. It is well-known that  weak convergence of $(x_n)$ or $(x_n^*)$ can be changed by weakly Cauchy convergence in this definition. A Banach space $X$ has the DPP if and only if any weakly compact operator $T: X \mapsto Y$ is completely continuous. Examples of spaces satisfying the DPP are $\mathsf{L}_1\left(\mu\right)$ and $C(K)$ for $K$ compact and Hausdorff. It is well-known that infinite dimensional reflexive spaces do not satisfy the DPP. A survey about the DPP can be found in \cite{D}. Recall that $X$ has the Schur property if for any sequence $x_{n}\overset{w }{\longrightarrow }0$ in $X$ we have that $\|x_n\| \rightarrow 0$ when $n \rightarrow \infty$. It is clear that $X$ has the DPP if $X$ has the Schur property.

Recall that a point $x$ in the unit sphere $\mathsf{S}_X$ of a Banach space $X$ is said to be a smooth point of $\mathsf{B}_X$ provided that there is only one functional in $\mathsf{S}_{X^*}$ attaining its norm at $x$. This unique functional is usually denoted by $\mathsf{J}_X\left(x\right)$. The set of smooth points of the (closed) unit ball $\mathsf{B}_X$ of a certain Banach space $X$ is usually denoted as $\mathrm{smo}\left(\mathsf{B}_X\right)$. A Banach space $X$ is said to be smooth provided that $\mathsf{S}_X=\mathrm{smo}\left(\mathsf{B}_X\right)$. Given a smooth Banach space $X$, the dual map of $X$ is defined as $\mathsf{J}_X: X \to X^*$ such that $\left\|\mathsf{J}_X\left(x\right)\right\|=\left\|x\right\|$ and $\mathsf{J}_X\left(x\right)\left(x\right)=\left\|x\right\|^2$ for all $x\in X$. It is well known that the dual map is $\left\|\cdot\right\|$-$w^*$ continuous and satisfies that $\mathsf{J}_X\left(\lambda x\right)=\overline{\lambda}\mathsf{J}_X\left(x\right)$ for all $\lambda\in \mathbb{C}$ and all $x\in X$. We refer the reader to \cite{DGZ} for a better perspective on these concepts.

\section{Dual maps and smoothness} \label{sec3}

Given any Banach space $X$, a map $T:X\setminus\{0\}\to X^*\setminus\{0\}$ is called a support mapping if $\|T(x)(x)\|=1=\|T(x)\|$ for all $x\in\mathsf{S}_X$ and $T(\lambda x)=\lambda T(x)$ for all $\lambda \geq 0$ and all $x\in X\setminus\{0\}$.

In \cite[Theorem 1, p. 22]{D1}, it was proved that a supporting map is norm to weak-star continuous from $S_X$ to $S_{X^*}$ at $x_0$ in $S_X$ if and only if $X$ is smooth at $x_0$. We will prove a more general version of this result by using dual maps instead of supporting maps. 

\begin{definition}
Given a Banach space $X$, a map $T:X \to X^*$ is called a dual map whenever
$\left\|T\left(x\right)\right\|=\left\|x\right\|$ for all $x\in X$ and $T\left(x\right)\left(x\right)=\left\|x\right\|^2$ for all $x\in X$. 
\end{definition}

Notice that dual maps are more general that supporting maps, as the following theorem states:
\begin{theorem}
Let $X$ be a Banach space.
\begin{enumerate}
\item If $T:X\setminus\{0\}\to X^*\setminus\{0\}$ is a supporting map, then by setting $T(0)=0$ we have that $T$ is a dual map.
\item If $X$ is not smooth, then there are dual maps on $X$ which restricted to $X\setminus\{0\}$ are not supporting maps.
\end{enumerate}
\end{theorem}

\begin{proof}
(1) is clear. To prove (2), assume that $X$ is not smooth and let $T:X\setminus\{0\}\to X^*\setminus\{0\}$ be any supporting map. Let $x\in\mathsf{S}_X$ be a non-smooth point and let $x^*\in\mathsf{S}_{X^*}\setminus \{T(x)\}$ such that $x^*(x)=1$. The map $S:X \to X^*$ such that $S(0)=0$, $S(2x)=2x^*$, and $S(z)=T(z)$ for all $z\in X\setminus\{0,2x\}$ is a dual map on $X$ whose restriction to $X\setminus\{0\}$ is not a supporting map. 
\end{proof}

The following properties are easy calculations. 
  
\begin{proposition}
Let $X$ be a Banach space and consider a dual map $T:X\to X^*$. Then:
\begin{itemize}
\item[a)] $T$ is $\left\|\cdot\right\|$-$\left\|\cdot\right\|$ continuous at $0$.
\item[b)] For every $x\in X\setminus \left\{0\right\}$, $\frac{T\left(x\right)}{\left\|x\right\|}$ belongs to $S_{X^*}$ and attains its norm at $\frac{x}{\left\|x\right\|}$.
\item[c)] If $x\in \mathrm{smo}\left(\mathsf{B}_X\right)$, then $T\left(x\right)=\mathsf{J}_X\left(x\right)$, so $T\left(\lambda x\right)=\overline{\lambda}T\left(x\right)$ for all $\lambda\in \mathbb{C}$.
\item[d)] If $X$ is smooth, then $T$ is the unique dual map for $X$.
\item[e)] If $X$ is not smooth, then there is a bijection between the set of dual maps and the following set: $$\prod_{x\in \mathsf{S}_X\setminus\mathrm{smo}\left(\mathsf{B}_X\right)} x^{-1}\left(1\right) \cap \mathsf{B}_{X^*}$$ given by $T\mapsto \left(T\left(x\right)\right)_{x\in \mathsf{S}_X\setminus\mathrm{smo}\left(\mathsf{B}_X\right)}.$
\item[f)] In general, there is a bijection between the set of dual maps and the following set: $$\prod_{x\in \mathsf{S}_X} x^{-1}\left(1\right) \cap \mathsf{B}_{X^*}$$ given by $T\mapsto \left(T\left(x\right)\right)_{x\in \mathsf{S}_X}.$
\end{itemize}
\end{proposition}

Hence, to construct a dual map, all is needed is mapping $0$ to $0$ and every non-zero $x$ to a norm-$1$ functional attaining its norm at $\frac{x}{\left\|x\right\|}$ times the norm of $x$.

\subsection{Points of $\left\|\cdot\right\|$-$w^*$ continuity}

Now we determine points of $\left\|\cdot\right\|$-$w^*$ continuity of dual maps, turning out to be exactly the smooth points. We will denote $x^{-1}(1):=\{ x^* \in X^* : x^*(x)=1 \}$. For this we will be making a strong use of the following  lemma:

\begin{lemma}\label{2}
Let $X$ be a $2$-dimensional real Banach space and $x\in \mathsf{S}_X\setminus \mathrm{smo}\left(\mathsf{B}_X\right)$. Let $x_1^*\neq x_2^*\in \mathsf{S}_{X^*}$ such that $x^{-1}\left(1\right)\cap\mathsf{B}_{X^*} =\left[x_1^*,x_2^*\right]$. There exist two sequences $(x_n), (y_n)\subset \mathrm{smo}\left(\mathsf{B}_X\right)$ satisfying:
\begin{enumerate}
\item $(x_n)$ and $(y_n)$ converge to $x$.
\item $(\mathsf{J}_X\left(x_n\right))$ and $(\mathsf{J}_X\left(y_n\right))$ converge to $x_1^*$ and $x_2^*$ respectively.
\end{enumerate}
\end{lemma}

To prove Lemma \ref{2}, we will need the following lemma:
\begin{lemma}\label{3}
Let $X$ be a real Banach space. Let $x\in \overline{\mathrm{smo}\left(\mathsf{B}_X\right)}$ such that $x^{-1}\left(1\right)\cap\mathsf{B}_{X^*}$ has non-empty interior relative to $\mathsf{S}_{X^*}$. Let $\left(x_n\right)_{n\in \mathbb{N}}\subset \mathrm{smo}\left(\mathsf{B}_X\right)$ be convergent to $x$ and consider $\left(\mathsf{J}_X\left(x_{n_i}\right)\right)_{i\in I}$ to be a subnet of $\left(\mathsf{J}_X\left(x_n\right)\right)_{n\in \mathbb{N}}$. Then:
\begin{enumerate}
\item If $\left(\mathsf{J}_X\left(x_{n_i}\right)\right)_{i\in I}$ is $w^*$-convergent to an element $x^*\in\mathsf{B}_{X^*}$, then $x^*\in x^{-1}\left(1\right)\cap\mathsf{B}_{X^*}$.
\item If $\left(\mathsf{J}_X\left(x_{n_i}\right)\right)_{i\in I}$ is convergent to an element $x^*\in\mathsf{B}_{X^*}$, then $x^*$ is in the boundary of $x^{-1}\left(1\right)\cap\mathsf{B}_{X^*}$ relative to $\mathsf{S}_{X^*}$.
\end{enumerate}
\end{lemma}

\begin{proof} To prove (1), notice that given any $\varepsilon >0$ there exists $i\in I$ such that $\left\|x_{n_i}-x\right\|<\frac{\varepsilon}{2}$ and $\left| \mathsf{J}_X\left(x_{n_i}\right)\left(x\right)-x^*\left(x\right)\right|<\frac{\varepsilon}{2}$. Then
\begin{eqnarray*}
\left|1-x^*\left(x\right)\right|&=&\left| \mathsf{J}_X\left(x_{n_i}\right)\left(x_{n_i}\right)-x^*\left(x\right)\right|\\
&\leq &\left| \mathsf{J}_X\left(x_{n_i}\right)\left(x_{n_i}\right)-\mathsf{J}_X\left(x_{n_i}\right)\left(x\right)\right|+ \left| \mathsf{J}_X\left(x_{n_i}\right)\left(x\right)-x^*\left(x\right)\right|\\
&\leq & \left\|x_{n_i}-x\right\| +  \left| \mathsf{J}_X\left(x_{n_i}\right)\left(x\right)-x^*\left(x\right)\right| < \varepsilon, 
\end{eqnarray*}
So $x^*\in x^{-1}\left(1\right)\cap\mathsf{B}_{X^*}$. 

To prove (2), notice that $x^*$ is not in the interior of $x^{-1}\left(1\right)\cap\mathsf{B}_{X^*}$ relative to $\mathsf{S}_{X^*}$ since all points in that interior are smooth points of $\mathsf{B}_{X^*}$. So if $x^*$ is one of those interior points, then there exists $i\in I$ so that $\mathsf{J}_X\left(x_{n_i}\right)$ is also an interior point and thus a smooth point of $\mathsf{B}_{X^*}$. This implies the contradiction that $x_{n_i}=x$.
\end{proof}

It is well known  that $\mathrm{smo}\left(\mathsf{B}_X\right)$ is dense in $\mathsf{S}_X$ provided that $X$ is separable (see \cite{DGZ}).
Also notice that if $X$ is real and $2$-dimensional and $x\in\mathsf{S}_X\setminus \mathrm{smo}\left(\mathsf{B}_X\right)$, then $x^{-1}\left(1\right)\cap\mathsf{B}_{X^*}$ is a non-trivial segment and thus it has non-empty interior relative to $\mathsf{S}_{X^*}$. \medskip

\textit{Proof of Lemma \ref{2}.} Under the assumptions of the lemma, we have that
$\mathrm{smo}\left(\mathsf{B}_X\right)\cap \mathsf{S}_X\cap \left(x_1^*-x_2^*\right)^{-1}\left(\left(0,+\infty\right)\right)$ and $\mathrm{smo}\left(\mathsf{B}_X\right)\cap \mathsf{S}_X\cap \left(x_1^*-x_2^*\right)^{-1}\left(\left(-\infty,0\right)\right)$ are dense in $\mathsf{S}_X\cap \left(x_1^*-x_2^*\right)^{-1}\left(\left(0,+\infty\right)\right)$ and $\mathsf{S}_X\cap \left(x_1^*-x_2^*\right)^{-1}\left(\left(-\infty,0\right)\right)$ respectively. We can then find two sequences $\left(x_n\right)_{n\in \mathbb{N}}\subset \mathrm{smo}\left(\mathsf{B}_X\right)\cap\mathsf{S}_X\cap \left(x_1^*-x_2^*\right)^{-1}\left(\left(0,+\infty\right)\right)$ and $\left(y_n\right)_{n\in \mathbb{N}}\subset \mathrm{smo}\left(\mathsf{B}_X\right)\cap\mathsf{S}_X\cap \left(x_1^*-x_2^*\right)^{-1}\left(\left(-\infty,0\right)\right)$ both converging to $x$. Because of the $w^*-$compacity of $\mathsf{S}_{X^*}$ we are entitled to assume without loss of generality that the two sequences $\left(\mathsf{J}_X\left(x_n\right)\right)_{n\in \mathbb{N}}$ and $\left(\mathsf{J}_X\left(y_n\right)\right)_{n\in \mathbb{N}}$ are respectively convergent to some $a^*\neq b^*\in \left\{x_1^*,x_2^*\right\}$ in accordance to Lemma \ref{3}. \qed \bigskip

Now we are on the right position to state and prove the main result in this section, which constitutes a generalization of \cite[Theorem 1, p. 22]{D1}.

\begin{theorem}
Let $X$ be a Banach space. Let $T:X\to X^*$ be a dual map and consider $x\in X\setminus \left\{0\right\}$. The following conditions are equivalent:
\begin{enumerate}
\item $T$ is $\left\|\cdot\right\|$-$w^*$ continuous at $x$.
\item $\frac{x}{\left\|x\right\|}\in \mathrm{smo}\left(\mathsf{B}_X\right)$.
\end{enumerate}
\end{theorem}

\begin{proof}
We will assume that $\left\|x\right\|=1$ without loss of generality.
\begin{enumerate}
\item[(2)$\Rightarrow$(1)] Assume first that $x\in \mathrm{smo}\left(\mathsf{B}_X\right)$. Let $\left(x_n\right)_{n\in\mathbb{N}}$ be a sequence in $X$ converging to $x$. Observe that $\left(T\left(x_n\right)\right)_{n\in \mathbb{N}}$ is a bounded sequence, therefore it has a $w^*$-convergent subnet $\left(T\left(x_{n_j}\right)\right)_{j\in J}$ to some element $f\in X^*$. Observe that $\left\|f\right\| \leq 1$ since $\left\|T\left(x_{n_j}\right)\right\| = \left\|x_{n_j}\right\| \to \left\|x\right\|=1$. Since
\begin{eqnarray*}
\left|T\left(x_{n_j}\right)\left(x\right)-1\right| &\leq& \left|T\left(x_{n_j}\right)\left(x\right)-T\left(x_{n_j}\right)\left(x_{n_j}\right)\right| + \left|T\left(x_{n_j}\right)\left(x_{n_j}\right)-1\right|\\
&\leq& \left\|T\left(x_{n_j}\right)\right\|\left\|x-x_{n_j}\right\| + \left| \left\|x_{n_j}\right\|^2-1\right|
\end{eqnarray*}
 for every $j\in J$, we deduce that $f\left(x\right)=1$. This means that $f=T\left(x\right)$ because $x\in \mathrm{smo}\left(\mathsf{B}_X\right)$. Following a similar reasoning one can show that every subnet of $\left(T\left(x_n\right)\right)_{n\in \mathbb{N}}$ has a further subnet $w^*$-converging to $T\left(x\right)$, which shows that $\left(T\left(x_n\right)\right)_{n\in \mathbb{N}}$ is $w^*$-convergent to $T\left(x\right)$.

\item[(1)$\Rightarrow$(2)] Conversely, assume that $T$ is $\left\|\cdot\right\|$-$w^*$ continuous at $x$. Suppose to the contrary that $x\in \mathsf{S}_X\setminus \mathrm{smo}\left(\mathsf{B}_X\right)$. Let $Y$ be a $2$-dimensional real subspace of $X$ contaning $x$. In accordance to Lemma \ref{2}, there exist two sequences $\left(x_n\right)_{n\in \mathbb{N}},\left(y_n\right)_{n\in \mathbb{N}}\subset \mathrm{smo}\left(\mathsf{B}_Y\right)$ both converging to $x$ and such that $\left(\mathsf{J}_Y\left(x_n\right)\right)_{n\in \mathbb{N}}$ and $\left(\mathsf{J}_Y\left(y_n\right)\right)_{n\in \mathbb{N}}$ converge to $x_1^*$ and $x_2^*$, respectively, where $x_1^*\neq x_2^*\in \mathsf{S}_{X^*}$ are such that $x^{-1}\left(1\right)\cap\mathsf{B}_{X^*} =\left[x_1^*,x_2^*\right]$. By hypothesis, $\left(T\left(x_n\right)\right)_{n\in \mathbb{N}}$ and $\left(T\left(y_n\right)\right)_{n\in \mathbb{N}}$ are both $w^*$-convergent to $T\left(x\right)$. Therefore, $\left(\mathrm{Re}\left( T\left(x_n\right)\right)|_Y\right)_{n\in \mathbb{N}}$ and $\left(\mathrm{Re}\left(T\left(y_n\right)\right)|_Y\right)_{n\in \mathbb{N}}$ are both $w^*$-convergent to $\mathrm{Re}\left( T\left(x\right)\right)|_Y$, which is impossible since $x_1^*\neq x_2^*$.
\end{enumerate}
\end{proof}

As a corollary, we obtain a new characterization of smoothness.

\begin{corollary}\label{jacalamaca}
Let $X$ be a Banach space. The following are equivalent:
\begin{enumerate}
\item There exists a $\left\|\cdot\right\|$-$w^*$ continuous dual map on $X$.
\item $X$ is smooth.
\end{enumerate}
\end{corollary}

\section{Dual maps and the Dunford-Pettis property} \label{sec4}

We show that dual maps and the Dunford-Pettis property interact in a bidirectional way.
We first prove that the Schur property can be characterized in two different ways in terms of the Dunford-Pettis property and the sequential $w $-$w $ continuity at $0$ of dual maps.

\begin{theorem}\label{joder}
Let $X$ be a Banach space. The following assertions are equivalent:
\begin{enumerate}
\item $X$ has the Schur property.
\item $X$ has the Dunford-Pettis property and every dual map is sequentially $w$-$w$ continuous at $0$.
\item $X$ has the Dunford-Pettis property and there exists a dual map that is sequentially $w$-$w$ continuous at $0$.
\end{enumerate}
\end{theorem}

\begin{proof}
Assume first that $X$ has the Schur property. It is well know that $X$ enjoys the Dunford-Pettis property. Let $\left( x_{n}\right) _{n\in \mathbb{N}}$ be a $w $-null sequence in $X$. By hypothesis, $\left( x_{n}\right)_{n\in \mathbb{N}}$ converges to $0$, so $\left( T\left(
x_{n}\right) \right) _{n\in \mathbb{N}}$ converges to $0$ and thus $\left( T\left( x_{n}\right) \right) _{n\in \mathbb{N}}$ is $w $-null, where $T:X\to X^*$ is any dual map. Conversely, assume that $X$ has the Dunford-Pettis property and there exists a dual map $T:X\to X^*$ that is sequentially $w$-$w$ continuous at $0$. Suppose to the contrary that $X$ does not enjoy the Schur property. We can find a sequence $\left( x_{n}\right) $ in $\mathsf{S}_{X}$ such that $x_{n}\overset{w }{\longrightarrow }0$ in $X$. By hypothesis, $T\left(x_{n}\right) \overset{w }{\longrightarrow }0$ in $X^{\ast }$. However, $T\left( x_{n}\right) \left( x_{n}\right) =1$ for every $n\in \mathbb{N}$, so we reach the contradiction that $X$ does not have the Dunford-Pettis property.
\end{proof}

\begin{corollary}\label{joder1}
Let $X$ be any separable Banach space which has the Dunford-Pettis property but not the Schur property. Then $X$ can be equivalently renormed in such way that it is smooth but its dual map is not sequentially $w $-$w $ continuous at $0$.
\end{corollary}

\begin{proof}
Notice that $X$ can be equivalently renormed to be smooth since it is separable (see \cite{Ma}). Theorem \ref{joder} assures that the dual map will not be continuous.
\end{proof}

\begin{example}
\textit{There exists smooth Banach spaces whose dual map fails to be $w$-$w$ continuous at $0$.} Indeed, in virtue of Corollary \ref{joder1} it only suffices to consider $c_0$ with an equivalent smooth norm.
\end{example}

Now we prove a sufficient (but not necessary) condition for a Banach space with the Dunford-Pettis property to have the Schur property. Again dual maps will play a crucial role. The main idea behind this is to version the following result (see \cite[Theorem 3]{D}).

\begin{theorem}[\cite{D}]\label{l2}
Let $X$ be a Banach space. If $X$ has the Dunford-Pettis property, then either $X$ contains a copy of $\ell _{\mathtt{1}}$ or $X^*$ has the Schur property.
\end{theorem}

The use of this result has been very important to study the connection between weakly compact and compact operators on spaces whose duals are Schur spaces. In particular, in spaces of weighted analytic functions (see \cite{M}).

By means of dual maps we will be able to provide a version of Theorem \ref{l2} with a very short proof. This result we are about to state and prove has been already proved in \cite[Theorem 3.16]{Ghenciu}. However, our proof is slightly different as we strongly rely on Rosenthal's $\ell _{\mathtt{1}}$-Theorem (see \cite{R}).

\begin{theorem}\label{l1}
Let $X$ be a Banach space. If $X$ has the Dunford-Pettis property, then either $X^{\ast }$ contains a copy of $\ell _{\mathtt{1}}$ or $X$ has the Schur property.
\end{theorem}

\begin{proof}
Suppose that $X$ fails to have the Schur property and consider a sequence $\left( x_{n}\right)_{n\in\mathbb{N}}$ in $\mathsf{S}_{X}$ such that $x_{n}\overset{w }{\longrightarrow }0$ in $X$. Let $T:X\to X^*$ be any dual map on $X$. Since $\left( T\left( x_{n}\right) \right)_{n\in\mathbb{N}} $ is bounded, we can apply Rosenthal's $\ell _{\mathtt{1}}$-Theorem (see \cite{R}) to find a subsequence $\left( x_{n_{k}}\right)_{k\in \mathbb{N}} $ such that $\left( T\left( x_{n_k}\right) \right)_{k\in\mathbb{N}} $ is either $w $-Cauchy in $X^{\ast }$ or equivalent to the $\ell_1$-basis. We will distinguish between these two possibilities:
\begin{itemize}
\item $\left( T\left( x_{n_k}\right) \right)_{k\in\mathbb{N}} $ is equivalent to the $\ell_1$-basis. In this case we inmediately have that $X^*$ has a copy of $\ell_1$ and so we are done.
\item $\left( T\left( x_{n_k}\right) \right)_{k\in\mathbb{N}} $ is $w $-Cauchy in $X^{\ast }$. In this situation, we have by hypothesis that $\left(T\left(x_{n_k}\right)\left(x_{n_k}\right)\right)_{k\in\mathbb{N}}$ converges to $0$, which contradicts that $T\left( x_{n}\right) \left( x_{n}\right) =\left\| x_{n}\right\|^{2}=1$ for every $n\in \mathbb{N}$.
\end{itemize}
\end{proof}

We obtain a well-known result as a corollary:
\begin{corollary}
If $X$ is an infinite dimensional Banach space and $X$ has the Dunford-Pettis property, then $\ell_1$ is a closed subspace of $X^*$.
\end{corollary}

\begin{proof} By Theorem \ref{l1}, if $\ell_1$ is not contained in $X^*$, then $X$ has the Schur property, so by Rosenthal's $\ell_1-$Theorem, $\ell_1$ is contained in $X$. It is well-known that in this case, $\ell_1$ is also contained in $X^*$ (see \cite{D3}, p.211), a contradiction.
\end{proof}

\bibliographystyle{amsplain}

\begin{thebibliography}{10}



\bibitem{DGZ}
R. Deville, G. Godefroy, and V. Zizler: Smoothness and Renormings in Banach Spaces. {\em Pitman Monographs and Surveys in Pure and Applied Mathematics}, \textbf{64}, Longman, Brunt Mill, (1993).

\bibitem{D}
J. Diestel: A survey of results related to the Dunford-Pettis property. Proceedings of the Conference on Integration, Topology, and Geometry in Linear Spaces (Univ. North Carolina, Chapel Hill, N.C., 1979), pp. 15--60, {\em Contemp. Math.}, {\bf 2}, Amer. Math. Soc., Providence, R.I., (1980).

\bibitem{D1} J. Diestel: {\em Geometry of Banach spaces-selected topics.} Lecture Notes in Mathematics {\bf 485}, Springer-Verlag, Berlin-New York, 1975.

\bibitem{D3} J. Diestel: {\em Sequences and series in Banach spaces.} Graduate Texts in Mathematics {\bf 92}, Springer-Verlag, New York, 1984. xii+261 pp.

\bibitem{Ghenciu}
I. Ghenciu; P. Lewis:  The Dunford-Pettis property, the Gelfand-Phillips property, and L-sets. {\em Colloq. Math.} {\bf 106} (2006), no. 2, 311--324.

\bibitem{Ma} S. Mazur: Über konvexe mengen in linearem normieten R$\ddot{a}$umen. {\em Studia Math.} {\bf 4} (1933), 70--84.

\bibitem{M} A. Miralles: Schur spaces and weighted spaces of type $H^\infty$. {\em Quaest. Math.} {\bf 35} (2012), no. 4, 463--470.

\bibitem{R} H. P. Rosenthal: A characterization of Banach spaces contaning $\ell_1$. {\em Proc. Nat. Acad. Sci. USA} {\bf 71} (1974), no. 6, 2411--2413.



\end{thebibliography}

\end{document}